\documentclass[11pt]{article}
\usepackage{times} 
\usepackage{amsmath}
\usepackage{amsthm}
\usepackage{latexsym}
\usepackage{amssymb}
\usepackage{enumerate}

\newcommand{\mr}{\mbox{mr}}

\DeclareMathOperator{\sign}{sign}
\newcommand{\reals}{\mathbb{R}}
\newcommand{\rationals}{\mathbb{Q}}

\DeclareMathOperator{\rank}{rank}

\newtheorem{theorem}{Theorem}

\newtheorem{example}[theorem]{Example}
\newtheorem{observation}[theorem]{Observation}
\newtheorem{conjecture}[theorem]{Conjecture}
\newtheorem{problem}[theorem]{Problem}
\newtheorem{lemma}[theorem]{Lemma}
\newtheorem{corollary}[theorem]{Corollary}
\title{  \ \\*[-0.25in] Minimum ranks of sign patterns via sign vectors and duality
}
\author{Marina Arav\footnote{Corresponding author, E-mail: marav@gsu.edu}, Frank J. Hall, Zhongshan Li, \\
 Hein van der Holst, John Sinkovic,  Lihua Zhang  \\
Department of Mathematics and Statistics \\
Georgia State University \\
Atlanta, GA 30303, USA
}
\date{}

\begin{document}

\maketitle

\begin{abstract}A {\it sign pattern matrix} is a matrix whose entries are from the set $\{+,-, 0\}$. The minimum rank of a sign pattern matrix $A$ is the minimum of the ranks of the real matrices whose entries have signs equal to the corresponding entries of $A$. It is shown in this paper that for any $m \times n$ sign pattern $A$ with minimum rank $n-2$, rational realization of the minimum rank is possible. This is done using a new approach involving sign vectors and duality. 
It is shown that for each integer $n\geq 9$, there exists a nonnegative integer $m$ such that there exists an $n\times m$ sign pattern matrix with minimum rank $n-3$ for which rational realization is not possible. A characterization of $m\times  n$ sign patterns $A$ with minimum rank $n-1$ is given (which solves an open problem in Brualdi et al.  \cite{Bru10}), along with a more general description of sign patterns with minimum rank $r$, in terms of sign vectors of certain subspaces. A number of results on the maximum and minimum numbers of sign vectors of $k$-dimensional subspaces of $\mathbb R^n$ are obtained. In particular, it is  shown that the maximum  number of sign vectors of $2$-dimensional subspaces of $\mathbb R^n$ is $4n+1$. Several related open problems are stated along the way. 

\end{abstract}

\medskip

  
\medskip

 
\bigskip

\medskip

\section{Introduction}

An important part of the combinatorial matrix theory is the study of sign pattern matrices, which has been the focus of extensive research for the last 50 years (\cite{Bru95}, \cite{Hall07}). A sign pattern matrix is a matrix whose entries are from the set $\{+, -, 0\}$. For a real matrix $B$, sign($B$) is the sign pattern matrix obtained by replacing each positive (respectively, negative, zero) entry of $B$ by $+$ (respectively, $-, 0$). For a sign pattern matrix $A$, the sign pattern class (also known as the qualitative class)  of $A$, denoted $Q(A)$, is defined as 
$$ Q(A)= \{ B : B \text{ is a real matrix and }\mbox{sign}(B) = A \}.$$

The minimum rank  of a sign pattern matrix $A$, denoted mr$(A)$, is the minimum of the ranks of the real matrices in $Q(A)$. Determination of the minimum rank of a sign pattern matrix in general is a longstanding open problem (see \cite{CJ82}) in combinatorial matrix theory. Recently, there has been a number of papers concerning this topic, for example \cite{A1}-\cite{Gra96},  \cite{Hall04}-\cite{Raz10}. In particular, matrices realizing the minimum rank of a sign pattern have applications in the study of neural networks \cite{Del} and communication complexity \cite{Raz10}.

More specifically, in the study of communication complexity in computer science, $(+, -)$ sign pattern matrices arise naturally; the minimum rank (also known as the \emph{sign-rank}) of the sign patterns plays an important role, as the minimum rank of a corresponding sign pattern matrix essentially determines the unbounded-error randomized communication complexity of a function.  New results on the minimum ranks of sign pattern matrices would have significant impact on the area of communication complexity, as well as several areas of mathematics, such as matroid theory, polytope theory, computational geometry, graph theory and combinatorics \cite{Gra96}.

In mathematics, rational realization is an important theme. For example, the study of the existence of rational (or integer) solutions of Diophantine equations is well-known. Steinitz's celebrated theorem stating that the  combinatorial type of every 3-polytope is rationally realizable is a far-reaching result \cite{Zieg}. In combinatorics, rational realizability of certain point-line configurations is an important problem (\cite{Grunbaum09, Sturmfels90}); this is closely related to the realizability of the minimum rank of a sign pattern with minimum rank 3 \cite{Jing13}.
More generally,  the rational realizability of the minimum rank of a sign pattern matrix is equivalent to the rational realizability of a certain point-hyperplane configuration. 


In \cite{A1}, several classes of sign patterns $A$ for which  rational realization of the minimum rank is guaranteed are identified, such as when $A$ is entrywise nonzero, or the minimum rank of $A$ is at most 2, or the minimum rank of $A$ is at least $n-1$, where $A$ is $m \times n$.
 It has been shown in \cite{KP}, through the use of a result in projective geometry, that rational realization of the minimum rank is not always possible.  Specifically, in \cite{KP}, the authors showed that there exists a $12\times 12$ sign pattern matrix  with minimum rank 3 but there is no rational realization of rank 3 within the qualitative class of the sign pattern. Independently, Berman et al. \cite{Ber} also provided an example of  a sign pattern for which the rational minimum rank is strictly greater than the  minimum rank over the reals.
Both of these papers use techniques based on matroids. More recently, Jing et al. \cite{Jing13} found a $9\times 9$ sign pattern matrix  with minimum rank 3 whose rational minimum rank is 4. 

We note that Li et al. \cite{Li13} showed that for every $n \times m$ sign pattern with minimum rank 2, there is an integer matrix in its sign pattern class each of whose entries has absolute value at most $2n-3$ that achieves the minimum rank 2.

One goal of this paper is  to show that for any $m \times n$ sign pattern $A$ with minimum rank $n-2$, rational realization of the minimum rank is possible. This is done using a new approach involving sign vectors and duality. 
Furthermore, it is shown that for each integer $n\geq 9$, there exists a nonnegative integer $m$ such that there exists an $n\times m$ sign pattern matrix with minimum rank $n-3$ for which rational realization is not possible.

Another goal is to use sign vectors of subspaces to investigate the minimum ranks of sign patterns further. In particular, a characterization of $m\times  n$ sign patterns $A$ with minimum rank $n-1$ is given, which solves an open problem posed in Brualdi et al. \cite{Bru10}.  We also obtain a characterization of L-matrices using sign vectors  and a more general description of sign patterns with minimum rank $r$, in terms of sign vectors of certain subspaces. A number of results on the maximum and minimum numbers of sign vectors of $k$-dimensional subspaces of $\mathbb R^n$ are obtained. In particular, it is  shown that the maximum  number of sign vectors of $2$-dimensional subspaces of $\mathbb R^n$ is $4n+1$. Several related open problems are also discussed.

\section{Sign vectors and duality}

For any vector $x\in \mathbb{R}^n$, we define the \emph{sign vector} of $x$, $\sign(x)\in \{+,-,0\}^n$, by 
\begin{equation*}
\sign(x)_i = \begin{cases}
+ \quad &\text{if }x_i>0,\\
0 \quad &\text{if }x_i=0,\\
- \quad &\text{if }x_i<0.
\end{cases}
\end{equation*}
For any  subspace $L\subseteq \mathbb{R}^n$, we define the \emph{set of sign vectors} of $L$ as
\begin{equation*}
\sign(L) = \{\sign(x)~|~x\in L\}.
\end{equation*}

\begin{observation} \label{Obs1} If $K$ and $L$ are  subspaces of $\reals^n$ with $\sign(K)=\sign(L)$, then $\dim(K) = \dim(L)$. 
\end{observation}

Indeed, consider the sign vectors of the columns of the  reduced column echelon form of a matrix whose columns form a basis of $K$. Such sign vectors are in $\sign(K)$ and hence also in $\sign(L)$. It follows that $ \dim(L)\geq \dim(K) $. By reversing the roles of $K$ and $L$, we get the reverse inequality. 

For a  subspace $L\subseteq \mathbb{R}^n$, as usual,  $L^\perp = \{x\in \mathbb{R}^n~|~x^Ty=0\text{ for all }y\in L\}$ denotes  the \emph{orthogonal complement} of $L$.

Two sign vectors $c, x \in \{+,-,0\}^n$ are said to be {\it orthogonal}, written as  $c \perp x $,  if  one of the following two conditions holds:
\begin{enumerate}
\item for each $i$, we have $c_i=0$ or $x_i=0$, or
\item there are indices $i,j$ with $c_i=x_i\not=0$ and $c_j=-x_j\not=0$.
\end{enumerate}
For a set of sign vectors $S\subseteq \{+,-,0\}^n$, the \emph{orthogonal complement} of $S$ is
\begin{equation*}
S^\perp = \{c\in \{+,-,0\}^n~|~c\perp x=0 \text{ for all }x \in S\}.
\end{equation*}
Notice that if $c,x \in \mathbb{R}^n$ and $c^Tx=0$, then $\sign(c)\perp  \sign(x)=0$.

We will use the following theorem, a proof of which can  be found in Ziegler \cite{Zieg}.

\begin{theorem}[Duality of oriented matroids]\label{thm:duality}
For any  subspace $L\subseteq \mathbb{R}^n$, 
$$\sign(L)^\perp=\sign(L^\perp).$$
\end{theorem}

A  subspace $L\subseteq \reals^n$ is called \emph{rational} if $L$ has a basis consisting of rational vectors.

\begin{lemma}\label{lem:rationalsign}
Let $L$ be a rational  subspace of $\ \reals^n$ and let $B = \begin{bmatrix} b_1 & b_2 & \ldots & b_k \end{bmatrix}$ be a matrix whose columns form a rational basis for $L$. Then for any $s\in \sign(L)$, there exists a rational vector $x\in \rationals^k$ such that $\sign(Bx) = s$.
\end{lemma}

\begin{proof}
Let $N = \{i~:~s_i = 0\}$ and let $V = \{z \in \reals^k~:~ (Bz)_i = 0 \text { for all } i \in N \}$. There exists a vector $y \in V$ such that $\sign(By) = s$. Since $\{b_1,b_2,\ldots,b_k\}$ is a rational basis for $L$, $V$ has a rational basis $\{c_1,c_2,\ldots,c_t\}$. Let $a = \begin{bmatrix} a_1 & a_2 & \ldots & a_t\end{bmatrix}^T \in \reals^t$ such that $a_1c_1 + a_2c_2 + \cdots + a_tc_t = y$. There exists a sequence of rational vectors $a(n) = \begin{bmatrix}  a_1(n) & a_2(n) & \ldots & a_t(n)\end{bmatrix}^T$, $n\in \mathbb{N}$ such that $\lim_{n\to\infty} a(n) = a$. Define $y(n) = a_1(n)c_1 + a_2(n)c_2 + \cdots +a_t(n)c_t$. Then $\lim_{n\to\infty} y(n) = y$ and each $y(n)$ is a rational vector. Since $\lim_{n\to\infty} B y(n) = B y$, there exists an $m$ such that $\sign(B y(m)) = \sign(B y) = s$. By letting $x = y(m)$, the  statement of the lemma is proved.
\end{proof}

\begin{theorem}\label{thm:equivspmlin}
Let $n$ and $r$ be nonnegative integers.
Rational realization of the minimum rank is always possible for all $n\times m$ sign pattern matrices with minimum rank $r$ if and only if for each  subspace $L\subseteq \reals^n$ with dimension $r$ there exists a rational subspace $K\subseteq \reals^n$ with $\dim(K) = r$ and $\sign(K) = \sign(L)$.
\end{theorem}
\begin{proof}
Suppose that rational realization of the minimum rank is always possible for all $n\times m$ sign pattern matrices with minimum rank $r$. Let $L\subseteq \reals^n$ with dimension $r$.
Let $s_1,\ldots,s_k$ be  all sign vectors in $\sign(L)$. Let $A$ be the $n\times k$ sign pattern matrix whose $i$th column is $s_i$. Then $\mr(A) \leq r$, as for each sign vector $s_i$ we can choose a vector $x_i\in L$ with $\sign(x_i)=s_i$, and the matrix $B$, whose $i$th column is $x_i$, has rank $r$. Thus, there exists a rational matrix $C\in Q(A)$ with rank at most $r$. Let $K$ be the column space of $C$. Then $\sign(K) = \sign(L)$ and $K$ is rational. Since $\sign(K) = \sign(L)$, by Observation \ref{Obs1}, $\dim(K) = r$.

Conversely, suppose that for each  subspace $L\subseteq \reals^n$ with dimension $r$ there exists a rational subspace $K\subseteq \reals^n$ with $\sign(K) = \sign(L)$. Let $A$ be an $n\times m$ sign pattern matrix with $\mr(A) = r$ and let $F\in Q(A)$ with $\rank(F)=r$. 
There exist an $n\times r$ matrix $U$ and an $r\times m$ matrix $V$ such that $F=U V$.
Let $v_i$ denote the $i$th column of $V$. Notice that $\sign(U v_i)$ is the $i$th column of $A$.
Let $L$ be the  subspace of $\reals^n$ spanned by the columns of $U$. By assumption, there exists a rational  subspace $K$ of $\reals^n$ such that $\sign(L) = \sign(K)$. Let $\{b_1,b_2,\ldots,b_r\}$ be a rational basis for $K$ and let $B = \begin{bmatrix} b_1 & b_2 & \ldots & b_r \end{bmatrix}$. Since $\sign(L) = \sign(K)$, there exist rational vectors $x_i \in \rationals^r$, $i=1,2,\ldots,m$ such that $\sign(B x_i) = \sign(U v_i)$, by Lemma~\ref{lem:rationalsign}. Let $C = B \begin{bmatrix} x_1 & x_2 & \ldots & x_m\end{bmatrix}$. Then $C \in Q(A)$, $\rank C\leq r$, and $C$ is a rational matrix. Since $\mr(A) = r$, $\rank C = r$. 
\end{proof}

The next theorem can be found in \cite{A1} and \cite{Li13}. 

\begin{theorem}\label{thm:rat2}
Rational realization of the minimum rank is always possible for every $n \times m$ sign pattern matrix $A$ with minimum rank $2$.
\end{theorem}

The following theorem follows from the two preceding theorems. 

\begin{theorem}\label{thm:dim2rat}
For any  subspace $L\subseteq \reals^n$ with $\dim(L)=2$, there exists a rational subspace $K\subseteq \reals^n$ with $\dim(K)=2$ and $\sign(K)=\sign(L)$.
\end{theorem}

Since the solution space of a system of homogeneous linear equations with rational coefficients has a rational basis, it is clear that if $L\subseteq \reals^n$ is a rational  subspace, then $L^\perp$ is also a rational subspace. 

\begin{lemma}\label{lem:rationaln-2}
For any subspace $L\subseteq \reals^n$ with $\dim L = n-2$, there exists a rational subspace $K\subseteq \reals^n$ with $\dim K = n-2$ such that $\sign(K) = \sign(L)$.
\end{lemma}
\begin{proof}
By Theorem~\ref{thm:dim2rat}, there exists a rational subspace $M\subseteq \reals^n$ with $\dim M = 2$ such that $\sign(L^\perp) = \sign(M)$. Let $K = M^\perp$. Then $K$ is a rational subspace of $\reals^n$ with $\dim K = n-2$ and, by Theorem~\ref{thm:duality}, $\sign(K) = \sign(M^\perp) = \sign(M)^\perp = \sign(L^\perp)^\perp = \sign(L)$.
\end{proof}

By Theorem~\ref{thm:equivspmlin}, and by considering the transpose if needed, we get the following result. 

\begin{theorem} \label{thm:mr=n-2}
Rational realization of the minimum rank is always possible for $n \times m$ sign pattern matrices with minimum rank $n-2$ or $m-2$.
\end{theorem}

From the $9\times 9$ example given in \cite{Jing13}, the next theorem follows immediately.

\begin{theorem}
For each integer $n\geq 9$, there exists an $n\times m$ sign pattern matrix $A$ with $\mr(A)=3$ for which no rational realization is possible. 
\end{theorem}

From Theorem~\ref{thm:equivspmlin} we then obtain the following corollary.

\begin{corollary}\label{cor:dim3}
Let $n\geq 9$ be an integer. Then there exists a  subspace $L\subseteq \reals^n$ with $\dim L = 3$ such that there is no rational  subspace $K\subseteq \reals^n$ with $\dim K = 3$ and $\sign(L)=\sign(K)$.
\end{corollary}

The following result then follows from Theorem~\ref{thm:duality}. 

\begin{lemma}\label{lem:rationaln-3}
Let $n\geq 9$ be an integer. 
There exists a  subspace $M\subseteq \reals^n$ with $\dim M = n-3$ such that there is no rational  subspace $K\subseteq \reals^n$ with $\dim K = n-3$ and $\sign(K) = \sign(M)$.
\end{lemma}

\begin{proof}
By Corollary~\ref{cor:dim3}, there exists a
 subspace $L\subseteq \reals^n$ with $\dim L = 3$ such that there is no rational  subspace $K\subseteq \reals^n$ with $\dim K = 3$ and $\sign(L)=\sign(K)$. Let $M = L^\perp$. Then there is no rational  subspace $K\subseteq \reals^n$ with $\dim K = n-3$ and $\sign(K) = \sign(M)$.
\end{proof}

Another application of Theorem~\ref{thm:equivspmlin} gives the following fact. 

\begin{theorem}
For each integer $n\geq 9$, there exists a nonnegative integer $m$ such that there exists an $n\times m$ sign pattern matrix with minimum rank $n-3$ for which rational realization is not possible.
\end{theorem}

However, this leaves open the following question.

\begin{problem}  Is it true that for each integer $n\geq 9$, there exists an $n \times n$ sign pattern matrix with minimum rank $n-3$ for which rational realization is not possible?
\end{problem}
 
Another natural question is the following. 
\begin{problem}  Is it true that for every 3 dimensional subspace $L$ of $\mathbb R^8$, there is a rational subspace $K$ such that sign$(L)$=sign$(K)$? 
\end{problem}

There are connections between rational realization of the minimum ranks of sign patterns and the existence of rational solutions of  certain matrix equations, as indicated in \cite{A3}. Using this connection and Theorem \ref{thm:mr=n-2} for the case of mr$(A)=n-2$, we are able to prove the following result.

 \begin{theorem} \label{thm:BC=0}
 Suppose that $B$, $C$ and $E$ are real matrices such that $BC=E$. If either $E$ has 2 rows or 2 columns, then there exist rational matrices $\tilde B$, $\tilde C$ and $\tilde E$ such that sign($\tilde B)$=sign$(B)$, sign($\tilde C)=$ sign$(C)$, sign($\tilde E)=$ sign$(E)$, and $\tilde B \tilde C =\tilde E$.
  \end{theorem}

\begin{proof}
Without loss of generality, assume that $E$ has two columns. Consider the $2 \times 2$ block matrix 
$$M= \begin{bmatrix}   
I_n & C \\
B & E 
\end{bmatrix}.$$ 
Observe that $M$ has $n+2 $ columns and the Schur complement of $I_n$ in $M$ is 0. Hence, rank$(M)=n$. It follows that the minimum rank of the sign pattern sign$(M)$ is $n$. Hence, from Theorem \ref{thm:mr=n-2}, there is a rational matrix 
$$\tilde M = \begin{bmatrix}   
D & C_1 \\
B_1 & E_1 
\end{bmatrix}$$ 
of rank $n$ in $Q(\mbox{sign}(M))$. 

It follows that the  Schur complement of $D$ in $\tilde M$ is $E_1-B_1D^{-1}C_1=0$. The rational matrices $\tilde B =B_1$,  $\tilde C= D^{-1} C_1$ and $\tilde E =E_1$
clearly satisfy the desired properties. 
\end{proof}

Since the zero entries of the matrix $E$ in a real matrix equation $BC=E$ are the major obstructions for the existence rational solutions within the same corresponding sign pattern classes, the preceding theorem suggests the following conjecture. 

\begin{conjecture}
Let  $B$, $C$ and $E$ be real matrices such that $BC=E$. If all the zero entries of $E$ are contained in a submatrix with either  2 rows or 2 columns, then there exist rational matrices $\tilde B$, $\tilde C$, and $\tilde E$, such that sign($\tilde B)$=sign$(B)$, sign($\tilde C)=$ sign$(C)$, sign($\tilde E)=$ sign$(E)$, and $\tilde B \tilde C =\tilde E$.
\end{conjecture}


\section{Further results on minimum ranks and  sign vectors of subspaces  }

We now investigate the minimum ranks of sign patterns further using sign vectors and duality. 

Similar to the concept of mr$(A)$, the {\it maximum rank} of a sign pattern matrix $A$, denoted  MR$(A)$,  is the maximum of
the ranks of the real matrices in $Q(A)$.  It is well-known that  
 MR$(A)$ is the maximum number of nonzero
 entries of $A$ with no two of the nonzero entries in the same row or in the same column. By a theorem of Konig, 
 the minimal number of lines (namely, rows and columns) in $A$ that cover all of the nonzero entries of 
 $A$ is equal to the maximal number of nonzero entries  in $A$, no two of which are on the same line. This common number is also called the {\it term rank} \cite{Hall07}. In contrast to the rational realization problem of the minimum rank, we note that through diagonal dominance, it can be easily seen that for every sign pattern matrix $A$, MR$(A)$ can always be achieved by a rational matrix. 

Sign pattern matrices $A$ that require a unique rank (namely, $ \mbox{MR}(A)=\mbox{mr}(A)$) were characterized by  Hershkowitz and  Schneider in \cite{Her93}. 

It is shown  in \cite{A2} that rational realization of the minimum rank of $A$ is always possible if MR$(A) -  \mr(A)=1$. The following conjecture from \cite{A2} remains open. 

\begin{conjecture}\cite{A2} \label{conj:gap2}
For every sign pattern matrix $A$ with MR$(A) - \text{mr}(A)=2$,  there is a rational realization of the minimum rank. 
\end{conjecture}

In \cite{A2}, the study of sign patterns $A$ with MR$(A) -  \mr(A)=1$ is reduced to the study of $m \times n$ sign patterns $A$ such that MR$(A)=n$ and mr$(A)=n-1$. 
In this connection,  using sign vectors and duality, we obtain the following characterization of the  $m\times n$ sign patterns $A$ such that MR$(A)=n$ and mr$(A)=n-1$ (which was raised as an open problem in \cite{Bru10}). Since the condition MR$(A)=n$ is easily checked, it suffices to characterize $m \times n$ sign patterns $A$ with mr$(A)=n-1$.

\begin{theorem}  \label{thm:mr=n-1} Let $A$ be an  $m\times n$ sign pattern and let row$(A)$ denote the set of the row vectors of $A$.  Then mr$(A)=n-1$ if and only if the following two conditions hold.
\begin{itemize}
\item[(a).] There is a nonzero sign vector $x\in \{+, -, 0\}^n$ such that $\text{row}(A)\subseteq x^\perp$. 

\item[(b).] For every two dimensional subspace $L$ of $\mathbb R^n$, $ \text{row}(A)\not\subseteq \text{sign}(L)^\perp$. 
\end{itemize}
\end{theorem}

\begin{proof} Assume that mr$({ A})=n-1$.  Then there is a real matrix $A_1 \in Q(A)$ such that rank$(A_1)=n-1$. Let $v_0$ be a nonzero vector in the null space  Null$(A_1)$ and let $x=$ sign$(v_0)$.  Then it is clear that  $\text{row}(A) \subseteq x^\perp$, namely, (a) holds. Further, consider any two dimensional  subspace $L$ of $\mathbb R^n$.   Suppose that 
$\text{row}(A) \subseteq  \text{sign}(L)^\perp$. Then by Theorem \ref{thm:duality},  $\text{row}(A)\subseteq \text{sign}(L^\perp)$. Since $L^\perp$ is a subspace of dimension $n-2$ and every row vector of $A$ is contained in $\text{sign}(L^\perp)$, it is clear that the span of the real vectors in $L^\perp $ whose sign vectors agree with the rows of $A$ is a subspace of dimension at most $n-2$, that is to say,  there is a real matrix $A_2 \in Q(A)$ of rank at most $n-2$, which contradicts the assumption that mr$(A)=n-1$.

Conversely, assume that (a) and (b) hold. Let $x_1$ be the $(1, -1, 0)$ vector in $Q(x)$. From (a), it is easily seen  that there is a real matrix $A_1\in Q(A)$ such that $A_1x_1=0$. It follows that rank$(A_1)\leq n-1$ and hence mr$(A)\leq n-1$. Suppose that  mr$(A)\leq  n-2$. Then there is a matrix $A_2\in Q(A)$ with rank$(A_2) \leq n-2$. It follows that  Null$(A_2)$ has dimension at least 2. Hence,  Null$(A_2)$ contains a subspace $L\in \mathbb R^n$ of dimension 2.  It follows that 
$\text{row}(A) \subseteq \text{sign}(\text{Null}(A_2))^\perp  \subseteq \text{sign}(L)^\perp$, contradiction (b). Thus   mr$({ A})=n-1$. 
\end{proof}

We note that for an $m\times n$ sign pattern matrix $A$, the condition that each row vector of $A$ is in $\{u, v\}^\perp$ for some two nonzero sign vectors $u, v\in \{+, -, 0\}^n$ with $u\neq \pm v$ does not imply that mr$(A)\leq n-2$, as the following example shows. 

\begin{example} Let $A=\bmatrix +&+&+ \\ 0 &+&+ \endbmatrix$, $u=\bmatrix +\\ + \\ - \endbmatrix$, and $v=\bmatrix 0\\ + \\ - \endbmatrix$. It is obvious that each row vector of $A$ is in $\{u, v\}^\perp$ for the two nonzero sign vectors $u, v\in \{+, -, 0\}^3$ with $u\neq \pm v$. But clearly, mr$(A)=2$, so  mr$(A)\not\leq n-2=1$.
\end{example}

Using a similar argument as in the proof of Theorem \ref{thm:mr=n-1}, we can show the following more general result that characterizes sign patterns with minimum rank $r$, for each possible $r\geq 1$. 

\begin{theorem}  \label{thm:mr=r} Let $A$ be an  $m\times n$ sign pattern and let $\text{row}(A)$ denote the set of the row vectors  of $A$, and let $r$ be a any integer such that $1 \leq r \leq \min \{m, n\}$.   Then mr$(A)=r$ if and only if the following two conditions hold.
\begin{itemize}
\item[(a).] There is a subspace $L\subseteq \mathbb R^n$ with $\dim (L) =r$ such that  $\text{row}(A) \subseteq \text{sign} (L)$. 

\item[(b).] For every subspace $V$ of $\mathbb R^n$ with $\dim (L) =r-1$, $ \text{row}(A) \not\subseteq \text{sign}(V)$. 
\end{itemize}
\end{theorem}

In particular, the above theorem gives a characterization of L-matrices (namely, $m \times n$ sign patterns $A$ with mr$(A)=n$). Note that every subset of $\{+, -, 0\}^n$ is of course contained in sign$(\mathbb R^n)$ and for every subspace $V$ of $\mathbb R^n$ of dimension $n-1$, sign$(V) = \text {sign}(x)^\perp$ for a nonzero vector $x$ in  $V^\perp$.

\begin{corollary}
\label{cor:mr=r} Let $A$ be an  $m\times n$ sign pattern and let $\text{row}(A)$ denote the set of the rows of $A$.   Then mr$(A)=n$ if and only if 
 for every nonzero sign vector $x \in \{+, -, 0\}^n$,   $ \text{row}(A) \not\subseteq  x^\perp$. 
\end{corollary}

The last characterization of L-matrices can be seen to be equivalent to the characterization of such matrices found in \cite{Bru95} (where a sign pattern  is defined to be an L-matrix if the minimum rank is equal to its number of rows). 

Let $S_{k, n}$ (respectively, $s_{k, n}$) denote the maximum cardinality (respectively, minimum cardinality)  of  sign$(L)$ as  $L$ runs over all $k$ dimensional subspaces of $\mathbb R^n$. In other words, 
$$S_{k, n} = \max_{\begin{smallmatrix} L \subseteq \mathbb R^n\\ \dim(L)=k \end{smallmatrix} }  |\text{sign}(L)|, \quad
s_{k, n} = \min_{\begin{smallmatrix} L \subseteq \mathbb R^n\\ \dim(L)=k \end{smallmatrix} }  |\text{sign}(L)|.
 $$
For every subspace $L$, as the nonzero vectors in sign$(L)$ occur in disjoint pairs of vectors that are negatives of each other, it is clear that $|\text{sign}(L)|$ is odd. Thus $S_{k,n}$ and $s_{k,n}$ are always odd. 
For each $k \ (0\leq k \leq n-1)$, since every $k$ dimensional subspace of $\mathbb R^n$ is contained in a subspace of dimension $k+1$, it is clear that
$S_{k, n} \leq S_{k+1, n}$ and $s_{k, n} \leq s_{k+1, n}$. 

Obviously, $S_{0, n}=s_{0, n}=1$, $S_{1, n}=s_{1, n}= 3$, and $S_{n, n} = s_{n,n}=3^n.$  For a nonzero sign vector $x\in \{+, -, 0\}^n$ with exactly $t$ nonzero entries, it can be seen that 
$| x^\perp |= 3^{n-t}(3^t- 2(2^t-1))$, which attains its  maximum $3^{n}-2(2^n-1)$ when $t=n$ and attains its minimum $3^{n-1}$ when $t=1$. Thus by the duality theorem, we get $S_{n-1, n}=3^{n}-2(2^n-1)$ and $s_{n-1, n} =3^{n-1}.$   In particular, for $n=3$, the above argument shows that the only possible cardinalities of $|\text{sign}(L)|$ for 2-dimensional subspaces of $\mathbb R^3$ are 9 and 13.

By considering the reduced row echelon form of a matrix whose rows form a basis for a $k$-dimensional subspace $L$ of $\mathbb R^n$ and observing that  the components in the pivot columns of the sign vectors  of the vectors $L$ are independent and arbitrary,  it can be seen that $s_{k, n} \geq 3^k$, for each $k, \ 1\leq k \leq n$.  For the $k$-dimensional subspace $L$ spanned by the standard vectors $e_1, e_2, \dots, e_k$, it can be seen that equality in the last inequality can be achieved. Thus,  $s_{k, n} = 3^k$, for each $k, \ 1\leq k \leq n$. 
We record some of these results below. 

\begin{theorem} Let $n\geq 2$. Then $S_{n-1, n}=3^{n}-2(2^n-1)$,    and $s_{k, n} = 3^k$, for each $k, \ 1\leq k \leq n$.
\end{theorem}

To determine $S_{2, n}$, we need some terminology and two results on sign patterns with minimum rank 2 given in \cite{Li13}. A sign pattern is said to be  \emph{condensed} if does not contain a zero row or zero column and no two rows or two columns are identical or are negatives of each other. Clearly, given any nonzero sign pattern $ A$, we can delete zero, duplicate or opposite rows and columns of $ A$ to get a condensed sign pattern matrix ${ A}_c$ (called the \emph{condensed sign pattern} of $ A$). Obviously, for any sign pattern $A$, mr$(A)=$mr$(A_c)$.

\begin{theorem} \cite{Li13} \label{thm:mr2char2} For any sign pattern $A$, mr$(A)=2$ if and only if $A_c$ has at least two rows, 
   each row  of ${ A_c}$ has at most one zero entry, and 
 there exist a  permutation sign pattern $P$  and a signature sign pattern $D$ such that each row  of  ${ A_c} DP$ is monotone (in the sense that $- \leq  0 \leq +$). 
\end{theorem} 

\begin{corollary} \label{cor:mr2} Let $\cal A$ be an $m\times n$ condensed sign pattern with minimum rank 2.  Then  $m\leq 2n$. 
\end{corollary}

We now determine the extremal cardinalities of sign$(V)$ for two dimensional subspaces $V$ of $\mathbb R^n\  (n\geq 2)$.

\begin{theorem}  For each $n\geq 2$, $S_{2, n}= 4n+1$.
\end{theorem}
\begin {proof} 

Let $V$ be any subspace  of $\mathbb R^n$ with dim$(V)=2$. Let $A$ be the sign pattern whose rows are the sign vectors in sign$(V)$.  Then clearly mr$(A)=2$. Note that $A$ has a zero row and the nonzero rows of $A$ occur in pairs of vectors that are negatives of each other. By deleting the zero row and choosing a representative out of each nonzero pair and possibly deleting some columns of $A$, we get the condensed sign pattern $A_c$ of $A$. By the preceding corollary, the number of rows of $A_c$ is at most $2n_1$, where $n_1$ is the number of columns of $A_c$. It is then clear that $| \text{sign}(V)|\leq 4 n_1 +1 \leq 4n+1$. 

By considering the  $2 n\times n$ sign pattern $B=\bmatrix T_1 \\T_2 \endbmatrix $, where 
$$T_1= \begin{bmatrix}  0 & + & \cdots & + \\ 
		- &  0  & \ddots & \vdots \\
		\vdots & \ddots & \ddots& + \\
		- & \cdots & - & 0 
\end{bmatrix},\quad \text{ and }\quad 
T_2 =
\begin{bmatrix}  - & + & \cdots & + \\ 
		- &  -  & \ddots & \vdots \\
		\vdots & \ddots & \ddots& + \\
		- & \cdots & - & - 
\end{bmatrix}, $$
it can be seen from Theorem \ref{thm:mr2char2} that mr$(B)=2$ and the equality $|\text {sign}(V)|= 4n+1$ can be achieved. 
\end{proof}

Due to the lack of a simple structural characterization of sign patterns with minimum rank 3 comparable to Theorem \ref{thm:mr2char2}, determining $S_{3, n}$ seems to be a difficult problem. We note that by considering the direct sum  $[+] \oplus B$, where $B$ is the $2(n-1)\times (n-1)$ sign pattern with mr$(B)=2$ as in the preceding proof, it can be seen that 
$$ S_{3, n} \geq 3(4(n-1)+1) = 3(4n-3).$$

\begin{problem} Find sharp bounds on $S_{k,n}$, for $3\leq k \leq n-2$. 
\end{problem}

\end{document}